%% file: Realizing-Ricci.tex
\documentclass[12pt]{article}
\usepackage[utf8]{inputenc}

\usepackage{graphicx}
\usepackage[all]{xy}
\usepackage{amsfonts}
\usepackage{amsmath}
\usepackage{amssymb}
\usepackage[thinlines]{easytable}
\usepackage{array}
\usepackage{enumerate}
\usepackage{tikz-cd}

\newcolumntype{M}[1]{>{\centering\arraybackslash}m{#1}}
\newcolumntype{N}{@{}m{0pt}@{}}

\def\squarebox#1{\hbox to #1{\hfill\vbox to #1{\vfill}}}
\def\lr#1{\left\langle #1\right\rangle}
\newtheorem{teorema}{Theorem}[section]
\newtheorem{lema}[teorema]{Lemma}
\newtheorem{corolario}[teorema]{Corollary}
\newtheorem{proposicao}[teorema]{Proposition}
\newtheorem{obs}[teorema]{Remark}

\newcommand{\qed}{\hspace*{\fill}	\vbox{\hrule\hbox{\vrule\squarebox{.667em}\vrule}\hrule}\smallskip}
\newenvironment{proof}{\noindent {\bf Proof:}}{\hfill $\qed $ \newline}

\newenvironment{exemplo}{\noindent {\bf Example:}}{}

\newenvironment{proof41}{\paragraph{Proof of Proposition \ref{prop:transposta}:}}{\hfill$\qed$}

\newenvironment{proof12}{\paragraph{Proof of Theorem \ref{ithm:convex}:}}{\hfill$\qed$}

\newcommand{\R}{\mathbb{R}}
\newcommand{\C}{\mathbb{C}}
\newcommand{\F}{\mathbb{F}}
\newcommand{\ad}{\mathrm{ad}}
\newcommand{\g}{\mathfrak{g}}
\renewcommand{\t}{\mathfrak{t}}

\renewcommand{\k}{\mathfrak{k}}
\newcommand{\m}{\mathfrak{m}}
\newcommand{\n}{\mathfrak{n}}
\newcommand{\h}{\mathfrak{h}}

\newcommand{\gl}{\mathfrak{gl}}
\renewcommand{\O}{\mathrm{O}}
\renewcommand{\prod}[1]{\langle {#1} \rangle}

\newcommand{\T}{\Theta}

\newenvironment{smallpmatrix}{\left(\begin{smallmatrix}}
	{\end{smallmatrix}\right)}

\DeclareMathOperator{\Ric}{Ric}
\DeclareMathOperator{\Ad}{Ad}
\DeclareMathOperator{\mrk}{mrk}

\newcommand{\ov}[1]{{\overline{#1}}}

\begin{document} 
	
	\title{On the embeddability of the homogeneous Ricci flow and its collapses}
	\author{Mauro Patrão, Lucas Seco, Llohann Sperança}
	\maketitle
	
	\begin{abstract}
	This article grew out of the urge to realize explicit examples of solutions for the Ricci flow as families of  isometrically embedded submanifolds, together with  its Gromov-Hausdorff collapses.
	To this aim, we consider the Ricci flow of invariant metrics in a class of flag manifolds. 
	
	On the one hand, we contrast with a previous result in literature by presenting entire flow lines of invariant metrics realized as orbits of a fixed representation. Indeed, we prove that
    the subset of realizable metrics has a global attractor with open interior,  containing an expressive family of complete flow lines.
	On the other hand, we prove that certain collapses cannot be realized in any fixed Euclidean space.  
    
    We provide a detailed picture of the flow, including examples of both realizable and non-realizable flow lines and collapses. 
    
	\end{abstract}
	

\bigbreak


	\section{Introduction}
	By {\em realizing} a Riemannian metric we mean obtaining it from an isometric embedding into a fixed Euclidean space. Because of their large isometry group, homogeneous manifolds $G/K$ of a Lie group $G$ are natural candidates for explicit realizations as orbits of representations.
	Nevertheless,  realizations are fated to limitations: although a compact family of invariant metrics on a fixed homogeneous manifold can always be realized as orbits of an isometric action \cite{moore1980equivariant}, here we prove that the Gromov-Hausdorff limit of a such a sequence is generically not realizable in the same fashion (see Theorem \ref{ithm:convergence} and the discussion in Section \ref{sec-collapse}).
	
    On the other hand, the invariance by the  Ricci flow with respect to several geometric properties of  (such as Positive Isotropic Curvature and non-negative bisectional curvature) are determinant ingredients in many important results (we refer to Wilking \cite{wilkinglie} for an account of such results and invariant conditions). However, the property of being realizable seems scarcely studied: \cite{safdari} (the only related work known to the authors) presents an example  of a hypersurface in Euclidean space whose intrinsic metric ceases to be realizable after a Ricci flow.
    
    Here we contrast with \cite{safdari} by presenting examples where the stronger condition of being realized by equivariant embeddings is invariant under the forward homogeneous Ricci flow. To this aim, we work with flag manifolds $G/K$ of a compact Lie group $G$, a family of manifolds that often  provide  solid working grounds  for computations (see for instance, \cite{anastassiou-chrysikos}, \cite{ricci} and \cite{ziller}).
    Flag manifolds are naturally realized in the Lie algebra $\g$ of $G$ as orbits of the adjoint representation.  Here we consider the action of $G$ in the $k$-fold product $\g^k$ defined  by the adjoint in each coordinate. We call such a representation as the \textit{adjoint representation in $\g^k$} and its orbits as \textit{adjoint orbits in $\g^k$}. We  explore how vast is the set of metrics and solutions of the homogeneous Ricci flow there realized. 
    
    For the realizable metrics, denote by $\mathcal{M}^{(k)}$ the subset of $G$-invariant metrics of the flag manifold $G/K$ which are realized as adjoint orbits in $\g^k$.
	We prove:
	
	\begin{teorema}\label{ithm:convex}
		Let $r$ be the rank of the flag manifold $G/K$ and $k \geq 1$. Then
		\[\mathcal{M}^{(k)}\subseteq \mathcal{M}^{(r)}=\rm{conv}\mathcal{M}^{(1)}\]
		where the latter denotes the convex hull of $\mathcal{M}^{(1)}$. Moreover, if $G/K$ has $T$-roots of type $A$, then $\mathcal{M}^{(r)}$ has non-empty interior.
	\end{teorema}

    For the realizable Ricci flow lines, we build-up on the results of \cite{ricci}, which studies the global dynamics and Gromov-Hausdorff collapses of the Ricci flow on flag manifolds with 3 isotropic summands by introducing the \textit{projected Ricci flow}, an alternative to the unit volume Ricci flow which we use here (see Section \ref{sec-prelim} for the appropriate definitions).
    When such flag manifolds have T-roots of type $A$, the projected Ricci flow has flow lines and equilibria as sketched in Figure \ref{JapanFlag}, where the phase space of metrics is the equilateral triangle $OPQ$
    (see Section 5 of \cite{ricci}). Such flag manifolds have rank 2, so that by Theorem \ref{ithm:convex} the relevant  metrics are $\mathcal{M}^{(2)}$, realizable as adjoint orbits in $\g^2$. We prove:
	
	   	\begin{figure}[h!]\label{JapanFlag}
		\begin{center}
			\def\svgwidth{10cm}
			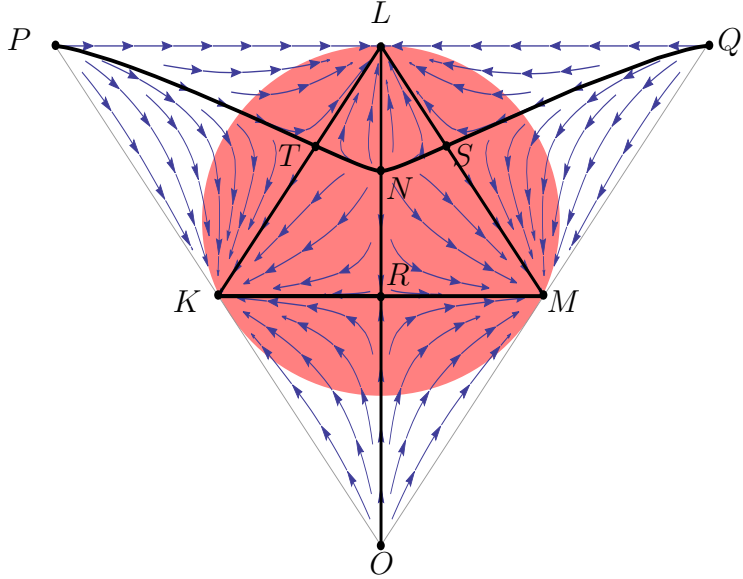
			\caption{The inscribed disc misses the collapsing limits at the vertices $O$, $P$, $Q$. But it contains the collapsing limits $K$, $L$, $M$ and all non-collapsing limits in its interior, besides the complete flow lines on and inside the triangle $KLM$.}
		\end{center}
	\end{figure}
	
	\begin{teorema}\label{ithm:figura}
	For the phase space of the projected Ricci flow of flag manifolds of 3 isotropic summands with T-roots of type $A$, we have that
	    \begin{enumerate}
	    \item The realizable metrics $\mathcal{M}^{(2)}$ equals the inscribed disk and is forward invariant by the projected Ricci flow. In particular, it contains the attractor given by the inscribed triangular region $KLM$.

	    \item Thus, every flow line contained in the disk can be realized continuously as adjoint orbits in $\g^2$, including the corresponding collapses at $K$, $L$, $M$.
	    
	    \item The collapses at $O,P,Q$ cannot be realized as equivariant isometric embeddings into a fixed Euclidean space.
	    \end{enumerate}
	\end{teorema}

The collapsing results of items 2 and 3 above are obtained by studying Hausdorff limits of equivariant realizations of a general homogeneous space $G/K$, we also comment on non-equivariant ones. In order to state our results, fix a $G$-invariant metric in $G/K$. Let $\k$ be the isotropy subalgebra and take the orthogonal decomposition $\g=\m\oplus\mathfrak k$. Identify $G$-invariant metrics in $G/K$  with the space of $\Ad(K)$-invariant inner products in $\m$. 
Now fix $G\to O(V)$ an orthogonal representation of $G$ in an Euclidean space $V$. Given $v \in V$, denote its isotropy subgroup by
	$$
	G_v = \{ g \in G:~ gv = v \}
	$$
so that the representation induces a $G$-equivariant realization $i: G/G_v \to Gv \subseteq V$ with the induced metric.


	
	\begin{teorema}\label{ithm:convergence}
		Let $g_n$ be a sequence of $K$-invariant inner products in $\m$ converging to $g$,  a positive semi-definite bilinear form in $\m$ and consider
$$
\ker g = \{X\in \mathfrak m~:~g(X,\mathfrak m)=0\}
$$
		Then,
		\begin{enumerate}
			\item The Gromov-Hausdorff limit of $(G/K,d_{g_n})$ does not depend on $g_n$, only on the limit $g$, where $d_{g_n}$ is the geodesic distance induced by $g_n$;
			\item If $v_n \in V$ converges to $v \in V$, put $K_n = G_{v_n}$, $K = G_v$ and let $g_n$ be induced by the  natural embedding $i_n:G/K_n \to Gv_n$. Then $(G/K_n,d_{g_n})$ converges to $(G/K,d_g)$ in the Gromov-Hausdorff sense;
			\item Suppose that $\ker g\oplus\mathfrak{k}$ is not a subalgebra and let $K_n = G_{v_n}$. Then, there is no sequence of isometric equivariant embeddings $i_n:(G/K_n,g_n)\to (V,\tilde g)$ such that  $g_n\to g$.
		\end{enumerate}
	\end{teorema}

	The structure of the article is as follows. In Section 2 we fix notation and conventions about flag manifolds and recall both the homogeneous  and projected Ricci flows. Section 3 is dedicated to study metrics induced by the adjoint representation and its products. In Section 4 we study the corresponding geometry to get Theorem \ref{ithm:convex}. In Section 5 we specialize to three isotropic summands and prove the first part of Theorem \ref{ithm:figura}.
	In Section 6 we go back to general homogeneous spaces, discuss metric limits and prove Theorem \ref{ithm:convergence} and then use it to prove the collapsing part of Theorem \ref{ithm:figura}.
	
	\section{Preliminaries}
	\label{sec-prelim}
	
	\subsection{The Ricci flow of invariant metrics}
	\label{general-ricci-flow}
	A family of Riemannian metrics $g(t)$ in a manifold $M$ is called a Ricci flow if it satisfies
	\begin{equation}
		\label{ricci-flow}
		\frac{\partial g}{ \partial t}=-2\Ric(g).
	\end{equation}
	
	Further suppose that $M$ is a compact homogeneous space with connected isotropy subgroup, $M=G/K$. Then, its set of $G$-invariant metrics is in one-to-one correspondence with $\Ad_G(K)$-invariant inner products at the basepoint $b=K$, obtained by restricting a metric $g$ to its value $g_b$ at $b$.  Furthermore, if $g$ is $G$-invariant, so it is its  Ricci tensor $\Ric (g)$ and  scalar curvature $S(g)$. Thus, they  are also  completely determined by their values at $b$, $\Ric(g)_b = \Ric(g_b)$, $S(g)_b=S(g_b)$. Taking this into account, the Ricci flow equation (\ref{ricci-flow}) becomes the autonomous ordinary differential equation known as the \textit{homogeneous Ricci flow}:
	\begin{equation}
		\label{inv-ricci-flow}
		\frac{dg_b}{dt}=-2\Ric(g_b).
	\end{equation}

	On the other hand, $(G/K,g)$  has  essentially the same geometry  if $g$ is rescaled   by a constant $\lambda > 0$. Moreover, 
	$\Ric(\lambda g) = \Ric(g)$. It follows that the Ricci operator $r(g)$, given by
	\begin{equation}
		\label{eq-ricciop}
		\Ric(g)(X, Y) = g( r(g)X, Y)
	\end{equation}
	is homogeneous of degree $-1$: $r(\lambda g) = \lambda^{-1} r(g)$. So does the scalar curvature $S(g) = \mathrm{tr}(r(g))$. 
	One can \textit{gauge away} the scale $\lambda$ by normalizing the flow. 
	For instance, if $M$ is compact, orientable with $\dim M = d$, one can consider the unit-volume normalized flow (see \cite{bohm}):
	\begin{equation}
		\label{normaliz-ricci-flow}
		\frac{dg_b}{dt}=-2\left( \Ric(g_b) - \frac{S(g_b) }{ d } g_b \right),
	\end{equation}
	which  is the gradient flow of $g_b\mapsto S(g_b)$, when $S$ is restricted to the unit-volume metrics.  
	The equilibria of (\ref{normaliz-ricci-flow}) are precisely the metrics satisfying $\Ric(g) = \lambda g$, $\lambda \in \R$, i.e., the so called {\em Einstein metrics}.  On the other hand, the unit volume Einstein metrics are precisely the critical points of the functional $S(g_g)$  on the space of unit volume metrics (see \cite{ziller}).  
	
	As in the unit-volume normalization \eqref{normaliz-ricci-flow}, one can normalize \eqref{inv-ricci-flow} by choosing an hypersurface in the (finite dimensional) space of homogeneous metrics which is transversal to the semi-lines $\lambda\mapsto \lambda g_b$. In the aforementioned case, the hypersurface consists on unit volume metrics and is unbounded. In order to study the limiting behavior of the Ricci flow, in Section \ref{preliminaries} we will normalize it, instead, to a simplex and rescale it to get a polynomial vector field, following \cite{ricci}.
	
	
	Recall that the submersion $G \to G/K$, $g \mapsto gb$, differentiates to the surjective map
	\begin{gather}\label{the-above-map}
	\g \to T_b(G/K) \qquad X\mapsto X \cdot b = d/dt (\exp(tX)b) |_{t=0}
	\end{gather}
	with kernel $\k$, the isotropy subalgebra. 
	Using that $g \in G$ acts in tangent vectors by its differential, we have that
	\begin{equation}
		\label{eq-induzido}
		g( X \cdot b ) = ( \Ad(g)X ) \cdot g b.
	\end{equation}
	In what follows we assume that the homogeneous space $M=G/K$ has connected isotropy $K$ and is \textit{reductive}, that is, $\g$ decomposes as $\mathfrak{g}=\mathfrak{k}\oplus\mathfrak{m}$ with $[\mathfrak{k},\mathfrak{m}]\subset \mathfrak{m}$. It follows that $\m$ is $\Ad_G(K)$-invariant. By  (\ref{eq-induzido}), we conclude that the restriction of \eqref{the-above-map} to $\m \to T_b(G/K)$ is a linear isomorphism that intertwines the isotropy representation of $K$ in $T_b(G/K)$ with the adjoint representation $\Ad_G(K)$  restricted to  $\m$.  This allows us to identify the $K$-isotropy representation on $T_b(G/K)$   with the  $\Ad_G(K)$-representation in $\m$.
	
	We further assume that $G$ is a  compact connected semisimple Lie group and that the isotropy representation of $G/K$ decomposes $\mathfrak{m}$ in $n$ summands
	\begin{equation}\label{deco-iso}
		\mathfrak{m}=\mathfrak{m}_1\oplus \ldots \oplus \mathfrak{m}_n,
	\end{equation}
	where $\mathfrak{m}_1,...,\mathfrak{m}_n$ are  irreducible pairwise non-equivalent isotropy representations.
	A source of examples satisfying the assumptions above are {\em flag manifolds of compact Lie groups} (see Section \ref{sec:flag} for details). With the assumptions above,  all invariant metrics are given by
	\begin{align}
		\label{eq-compon-metr}
		g_b&=x_1B_1+\ldots + x_nB_n, 
	\end{align}
	where $x_i>0$ and $B_i$ is the restriction of the (negative of the) Cartan-Killing form of $\mathfrak{g}$ to $\mathfrak{m}_i$. We also have that $r_k=r(g)|_{\mathfrak{m}_k}$ is a multiple of the identity, since it commutes with the isotropy representation. Therefore, 
	\begin{align}
		\label{eq-compon-ricci}
		\Ric (g_b)&=y_1 B_1 + \ldots + y_nB_n
	\end{align}
	where 
	\[y_i=x_ir_i.\]
	Therefore, the Ricci flow equation (\ref{inv-ricci-flow}) becomes the autonomous system of ordinary differential equations
	\begin{equation}
		\label{eq-ricci-flow-coords}
		\frac{dx_k}{dt}= -2 r_kx_k, \qquad \ k=1,\ldots , n.
	\end{equation}

	\subsection{Projected Ricci flow}\label{preliminaries} 
	
	Here we gather some information about the projected Ricci flow approach introduced in \cite{ricci}. Let $R(x)$ be the vector field of the Ricci flow (\ref{eq-ricci-flow-coords}), for short {\em the Ricci vector field}, which is a rational function of $x$ and homogeneous of degree $0$. Given $x = (x_1, \ldots, x_n) \in \mathbb{R}_+^n$, $x$ and $\lambda x$ describe essentially the same geometry. In order to take advantage of this and  the rationality of $R(x)$, we normalize the latter  to a simplex and rescale it to get a polynomial vector field.  
	
	Consider the linear scalar function
	$$
	\overline{x} = x_1 + \cdots + x_n,
	$$
	whose level set $\overline{x} = 1$ in  $\R^n_+$ is the open canonical $n$-dimensional simplex ${\mathcal T}$.
	
	\begin{proposicao}[Corollary 4.3, \cite{ricci}]
		\label{thm-rescaling}
		The solutions of
		\begin{equation}
			\label{eq-R1}
			\frac{d x }{dt}= R(x), \qquad x \in \R^n_+
		\end{equation}
		can be rescaled in space and reparametrized in time to solutions of
		\begin{equation}
			\label{eq-R-projetada}
			\frac{d x }{dt}= R(x) - \ov{R(x)}x, \qquad x  \in {\mathcal T}
		\end{equation}
		and vice-versa.	Furthermore, if $x\in\mathcal T$, then $R(x) = \lambda x$ for some  $\lambda \in \R$  if and only if $x$ is an equilibrium for (\ref{eq-R-projetada}).
		
	    Moreover, there exists a function which is strictly decreasing on non-equilibrium solutions of the normalized flow \eqref{eq-R-projetada}. In particular, the projected Ricci flow does not have non-trivial periodic orbits.
	\end{proposicao}

	\begin{obs} The unit-volume normalization \eqref{normaliz-ricci-flow} can be obtained in an analogous way, by replacing the functional $\bar x$ with $vol(x)$, the latter being  positive and homogeneous of degree $d/2$ in $x$. Indeed, in this case,  $\ov{R(x)}$ in \eqref{eq-R-projetada} is replaced by the multiple of the scalar curvature $-2 S(x)/d$,  recovering \eqref{normaliz-ricci-flow}.
 	\end{obs}
	
	
To study the limiting behavior of \eqref{eq-R-projetada} on ${\cal T}$, 
it is convenient to multiply it by an appropriate positive function $f: \mathbb{R}_+^n \to \mathbb{R}_+$
in order to get a homogeneous polynomial vector field $X(x)$ defined in the closure of ${\cal T}$ and tangent to the boundary of $\mathcal T$, given by
	\begin{eqnarray}
	\label{campo-X}
	X(x) & = & f(x)\left (R(x) - \overline{R(x)}\, x \right ) \label{eq-def-X} \\
	& = & (fR)(x) - \overline{(fR)(x)}\, x \nonumber
	\end{eqnarray}
since $\overline{x}$ is linear.  In particular, solutions of the new field in the interior of $\mathcal T$ are time-reparametrizations of \eqref{eq-R-projetada}. 
Therefore, to get a polynomial vector field $X$, it suffices to choose $f$ such that $(fR)(x) = f(x) R(x)$ is a polynomial vector field.
Moreover, in order for $X$ to be tangent to the boundary of $\mathcal T$, it is sufficient that the $i$-th coordinate of $(fR)(x)$ vanishes whenever the $i$-th coordinate does or, equivalently, that each coordinate hyperplane $\Pi_i = \{ x:~x_i = 0\}$ is invariant by the flow of $fR$.  Given a subset of indexes $I \subseteq \{ 1, \ldots, n \}$, consider the subspace $\Pi_I = \cap_{i \in I} \Pi_i$ and let ${\cal T}_I = \mathrm{cl}({\cal T}) \cap \Pi_I$ be the $I$-th face of the simplex ${\cal T}$. 
Note that ${\cal T}_\varnothing = \mathrm{cl}({\cal T})$.


	\begin{proposicao}[Proposition 4.3, \cite{ricci}]
		\label{propos-estratific}
		If $fR$ is tangent to each hyperplane $\Pi_i$, then each face ${\mathcal T}_I$ of ${\mathcal T}$ is invariant by the flow of $X$.  In particular, $\mathrm{cl}({\mathcal T})$ is invariant and its vertices are fixed points. 
	\end{proposicao}
	
	Instead of analyzing the dynamics of  $X$, it is more convenient to analyze its dynamics on the projected simplex
	\[
	{\mathcal S} = \{(x_1,\ldots,x_{n - 1}) \in \mathbb{R}_+^{n-1} :~ x_1 + \cdots + x_{n - 1} \leq 1\}
	\]
	associated to the conjugated vector field
	$
	Y = P \circ X \circ P^{-1},
	$
	where $P: {\mathcal T} \to {\mathcal S}$ is given by the projection
	$P(x_1,\ldots,x_{n - 1},x_n) = (x_1,\ldots,x_{n - 1})$.
	The flow of $Y$ in ${\mathcal S}$ is the so called {\em projected Ricci flow}. We conclude:
	
	\begin{proposicao}
		\label{prop:simplex-projection}
		If the vector field $fR$ is polynomial of degree $d$, then the vector fields
		$X$ given by equation \eqref{campo-X} and
		$Y = P \circ X \circ P^{-1}$
		are polynomial of degree $d+1$ and the associated flows are conjugated.  Furthermore, $x \in {\mathcal T}$ is Einstein if and only if $Y( Px ) = 0$.
	\end{proposicao}
	
	
	

	\subsection{Flag manifolds}\label{sec:flag}
	
	Here we recall some results and notations about compact Lie groups and its flag manifolds (for details and proofs see, for example, \cite{helgason}). Let $G$ be a compact connected Lie group equipped with  a $\Ad(G)$-invariant inner product $( \cdot, \, \cdot )$ in its Lie algebra $\g$. Fix a maximal torus $T\subseteq G$ with Lie algebra $\t$.  It is known that $\g$ is the compact real form of the complex reductive Lie algebra $\g_\C$, the complexification of $\g$.  The adjoint representation of the Cartan subalgebra $\h = \t_\C$ splits as the root space decomposition
	$
	\g_\C = \h \oplus \sum_{\alpha \in \Pi} \g_\alpha
	$, where
	$$
	\g_\alpha = \{ X \in \g_\C :\, \ad(H) X =  \alpha(H) X, \, \forall H \in \h  \},
	$$
	for a root system $\Pi \subset \h^*$.  
	Let $\Pi^+$ be a choice of positive roots corresponding to a choice of Weyl chamber $\t^+ \subseteq \t$. Then $\g$ splits as the orthogonal sum
	$$
	\g = \t \oplus \sum_{\alpha \in \Pi^+} \m_\alpha,
	$$
	where $\m_\alpha$ is the real root space
	$$
	\m_\alpha = \g \cap ( \g_\alpha \oplus \g_{-\alpha} ).
	$$
	
	An homogeneous space $G/K$ is called a  \textit{flag manifold} if $K$ is the centralizer of a (non-necessarily maximal) torus in $G$. One concludes that  $K$ is connected and w.l.o.g.\ we may assume that $T \subset K$.  Generally,  one can choose a subset of simple roots $\T$ and consider $K = G_\T$,  the centralizer of the abelian subalgebra
	$$
	\t_\Theta = \{ H \in \t:~ \alpha(H) = 0, \, \alpha \in \Theta \}.
	$$
	Roughly,  $\T$ furnishes the block structure of the isotropy $G_\T$, which characterizes the corresponding flag manifold. 
	Denote by $\lr \Theta $ the set of roots spanned by $\Theta$ and let $\langle \Theta \rangle^+ = \langle \Theta \rangle \cap \Pi^+$.
	Then the isotropy Lie algebra $\k = \g_\T$ splits as
	\begin{equation}
	\label{eq-isotropy-flag}
	\g_\T = \t \oplus \sum_{\alpha \in \langle \T \rangle^+} \m_\alpha,
	\end{equation}
	Finally, we denote the flag manifold
	\begin{equation}
		\label{eq-def-flag}
		\F_\T = G/G_\T.
	\end{equation}
	We fix the basepoint $b=G_{\T}$ and observe that the orthogonal complement of $\g_\T$ is 
	$$
	\m_\T \, = \sum_{\alpha \in \Pi^+ - \langle \T \rangle} \m_\alpha,
	$$
	which is naturally  $G_\T$-invariant.  It follows that $\F_\T$ with
	$$
	\g = \g_\T \oplus \m_\T
	$$
	is reductive and the isotropy representation of $\F_\T$ is equivalent to the adjoint representation of $G_\T$ in $\m$. 
	
	Since the center $Z$ of $G$ is contained in $T$, it follows that $Z$ is contained in $G_\T$. Taking the quotient of both $G$ and $G_\T$ by $Z$ in (\ref{eq-def-flag}), we obtain the same flag manifold. Note that $G/Z$ is isomorphic to the adjoint group of $\g$.  Thus, $\F_\T$ depends only on the Lie algebra $\g$ of $G$, which we can assume to be semisimple, and on the subset of simple roots $\T$.  By definition, the rank of the flag manifold $\F_\T$ coincides with $\dim \t_\Theta = \#( \Sigma - \Theta)$.
	
	\subsection{Isotropy decomposition}
	
	Since $G_\T$ is connected, its adjoint representation in $\m$ is completely determined by the adjoint representation of $\g_\T$ in $\m$.
	The latter was completely described by \cite{siebenthal} who showed that the real root space  of $\t_\T$ in $\m$ decomposes as a sum of non-equivalent irreducible $\g_\T$-representations.
	
	Denote the restriction of a root $\alpha \in \Pi$ to  $\t_\T$ by
	\begin{equation}
		\label{def-restricao}
		[\alpha] := \alpha|_{\t_\T}. 
	\end{equation}
	We say that two roots $\alpha, \, \beta$ are $\Theta$-equivalent if $[\alpha] = [\beta]$. It follows that
	$$
	\Pi^+_\T := [\Pi^+ - \langle \T \rangle]
	$$
	is exactly the set of real roots that appear in the adjoint representation of $\t_\T$ on $\m$, with corresponding real root spaces given by
	$$
	\m_{[\alpha]} = \sum_{ [\beta] = [\alpha] } \m_\beta.
	$$
	It follows that
	$$
	\m_\Theta =  \sum_{ [\alpha] \in \Pi^+_\T } \m_{[\alpha]}
	$$
	is a sum of non-equivalent irreducible $\g_\T$-representations.
	
	Note that, in general, the functionals $\Pi^+_\T$ of $\t_\T$ appearing in the above isotropy decomposition do not determine a root system: they are the so called {\em $T$-roots} of a flag manifold and have only being classified for flag manifolds with a small number of $\m_{[\alpha]}$-summands so far (see for example \cite{Arvanitoyeorgos-Chrysikos-Sakane-2013}).  
	Recall that the Dynkin mark of a simple root $\alpha\in\Sigma$ is the coefficient $\mrk(\alpha)$ of $\alpha$, in the expression of the highest root as a combination of simple roots.
    It follows  that the coefficient of $\alpha$ in any root is less then $\mrk (\alpha)$.
	Here we will use the following classification (see \cite{anastassiou-chrysikos2,kimura} for details and corresponding flag manifolds).
	
	\begin{table}[h!]
		\caption{\label{tab-upto3} Classification of $T$-roots up to three isotropy summands}
		\hspace{-1.8cm}
		\begin{tabular}{llllccccr}
			Roots && $\Theta$-classes && summands && rank && name \\ \hline \hline
			$\Sigma \setminus \Theta = \{\alpha\}$, $\mrk(\alpha)=1$ && $[\alpha]$ && 1 && 1 
			&& Type I
			\\ \hline 
			$\Sigma \setminus \Theta = \{\alpha\}$, $\mrk(\alpha)=2$ && $[\alpha], \, 2[\alpha]$ && 2 && 1
			&& Type I
			\\ \hline 
			$\Sigma \setminus \Theta = \{\alpha\}$, $\mrk(\alpha)=3$ && $[\alpha], \, 2[\alpha], 3[\alpha]$ && 3 && 1 
			&& Type I
			\\ \hline 
			$\Sigma \setminus \Theta = \{\alpha, \beta\}$, $\mrk(\alpha)=\mrk(\beta)=1$ && $[\alpha], \, [\beta], [\alpha+\beta]$ && 3 && 2 
			&& Type II
			\\ \hline 
		\end{tabular}
	\end{table}
	
	We say that the set of  $T$-roots $\Pi^+_\T$ is of type $A,~B,~C,~D,~E$ or $F$ if $\Pi^+_\T$ is isomorphic to the respective type of system  in the classification of irreducible root  systems.  For instance, all three summands flag manifolds of Type II above have $T$-roots of type $A$ (more precisely, of type $A_2$).
	
	From Section \ref{general-ricci-flow}, it follows that the spaces of $G$-invariant semi-definite metrics on $\F_\T$ are completely parametrized by the orthant
	$$
	\mathcal{M}_\T = \{ 
	(x_{[\alpha]})_{[\alpha]}\in \R^{\Pi^+_\T}:\, x_{[\alpha]} > 0, \, 
	[\alpha] \in \Pi^+_\T 
	\}
	$$
	and its closure
	$$
	\overline{\mathcal{M}_\T} = \{ 
	(x_{[\alpha]})_{[\alpha]}:\, x_{[\alpha]} \geq 0, \, 
	[\alpha] \in \Pi^+_\T 
	\} .
	$$
	
	Next we explore the most natural embedding of the flag manifold $\F_\T$ and the corresponding subset of $\mathcal{M}_\T$ thus realized.
	
	
	%
	

	\section{Metrics induced by the adjoint representation}

	A flag manifold $\F_\T = G/G_\Theta$ can be embedded as the adjoint orbit of  an element $H \in \t_\T$, as follows. Consider the centralizer $G_H$ in $G$ and the corresponding homogeneous space
	$$
	\F_H = G/G_H.
	$$
	Since $H \in \t_\T$, it follows that $G_\T \subseteq G_H$, giving a natural $G$-equivariant projection
	$$
	\F_\Theta \to \F_H
	$$
	which furnishes the collapse phenomena we will be interested in (Theorem \ref{ithm:convergence}).  
	We have that  $G_H$ is connected and with Lie algebra
	\begin{equation}
		\label{eq-u_H}
		\g_H = \t \oplus \sum_{\alpha \in \Pi^+,\, \alpha(H)=0} \m_\alpha.
	\end{equation}
	By \eqref{eq-isotropy-flag} it follows that $G_H = G_\T$ if, and only if, the set of roots that annihilate $H$ equals the set of roots spanned by $\T$, namely, if $H$ lies in
	$$
	\t_\Theta^\circ = \{ H \in \t_\Theta: \, \alpha(H) \neq 0, \quad \forall \alpha \in \Pi^+  - \langle \Theta \rangle \}.
	$$
	Note that $\t^\circ_\T$  is the complement of finite number of hyperplanes in $\t_\T$,  and thus open and dense in $\t_\Theta$.
	
	In conclusion, the flag manifold $\F_\T = G/G_\T$, with basepoint $b_\T = G_\T$, can be embedded in $\g$ as follows: let $H \in \t_\Theta^\circ$ and consider the embedding induced by the adjoint orbit: 
	\begin{equation}
		\label{eq-mergulho}
		f_{H}: 
		\F_\Theta \to \g, \qquad u b_\Theta \mapsto \Ad(u)H.
	\end{equation}
	Although this is a natural realization, we will see that the class of metrics realized by this embedding is very limited. 
	
	In order to compute the metric induced by the embedding, note that, for $X \in \g$ and $u\in G$, 
	\begin{equation}
		\label{eq-derivada-mergulho}
		f'_{H}(u b_\Theta) u(X \cdot b_\Theta) = \Ad(u)[X, H].
	\end{equation}
	Fix $( \cdot, \, \cdot )$ an $\Ad(G)$-invariant inner product in $\g$. 
	Note that $( \cdot, \, \cdot )$ is a multiple of the Cartan-Killing form in each simple component of $\g$.  Then the embedding above induces the $G$-invariant Riemannian metric $\lr{,}$ in $\F_\Theta$ given by
	$$
	\langle u( X \cdot b_\Theta),\, u (Y \cdot b_\Theta ) \rangle_{u b_\Theta} 
	=
	\left( f'_{H}(u b_\Theta) u(X \cdot b_\Theta), \, 
	f'_{H}(u b_\Theta) u(Y \cdot b_\Theta ) \right).
	$$
	By equation (\ref{eq-derivada-mergulho}) and $\Ad(G)$-invariance of $(\, ,)$,
	\begin{equation}
		\label{eq-metrica-inv}
		\langle u( X \cdot b_\Theta),\, u (Y \cdot b_\Theta ) \rangle_{u b_\Theta} 
		=            
		\left( [X, H], \,  [Y, H] \right).
	\end{equation}           
	
	On the other hand,  $\g = \g_\Theta \oplus \m_\Theta$, where
	\begin{eqnarray}
		\g_\Theta   =  \t \oplus \sum_{\alpha \in \langle \Theta \rangle^+} \g_\alpha,\qquad 
		\label{eq-esp-tg}
		\m_\Theta  =  \sum_{\alpha \in \Pi^+ - \langle \Theta \rangle} \g_\alpha.
	\end{eqnarray}
	Since $H \in \t$, we have that $\ad(H)$ is skew-symmetric with respect to $( \cdot, \, \cdot )$ and leaves $\g_\alpha$ invariant.  It follows that  $\ad(H)^2$ acts in $\g_\alpha$, hence in $\m_\alpha$, as multiplication by $-\alpha(H)^2$.
	In particular, for $X, Y \in \m_\Theta$, we write  $X = \sum_\alpha X_\alpha$,  $Y = \sum_\alpha Y_\alpha$, with $X_\alpha, Y_\alpha \in \m_\alpha$ and conclude that
	\begin{eqnarray}
		\langle X \cdot b_\Theta,\, Y \cdot b_\Theta \rangle_{b_\Theta} 
		& =  &          
		\left( \ad(H)X, \,  \ad(H)Y \right) \nonumber \\
		& =  &          
		- \left( \ad(H)^2 X, \,  Y \right) \nonumber \\
		& =  &          
		\sum_\alpha \alpha(H)^2 \left( X_\alpha, \,  Y_\alpha \right) .
		\label{eq-metrica-mergulho}
	\end{eqnarray}
	
	If  a general $H \in \t_\Theta$ is allowed, then the metric  can become semi-definite, thus defining the map:
	\begin{equation}
		\label{eq-metrica-mergulho-map}
		\mu^{(1)}: \t_\Theta \to \overline{\mathcal{M}_\T}
		\qquad
		H \mapsto (\alpha(H)^2)_{[\alpha]}.
	\end{equation}
	Note that  $[\alpha]=[\beta]$ guarantees that $\alpha(H) = \beta(H)$ for every $H \in \t_\T$.
	
	Equation \eqref{eq-metrica-mergulho-map} induces strong restrictions on  metrics realizable as adjoint orbits. For example, for the class of 3-summands flag manifolds of Type II we will show in Section \ref{sec:realizing} that the image of $\mu^{(1)}$ is the cone with vertex at the origin generated by the boundary of the disc in Figure \ref{JapanFlag}.
	One can naturally improve the situation by taking convex sums of these metrics, thus realizing also the interior of the cone. Interestingly enough, such a sum coincides with the orbit metric in a product of adjoint representations. 
	
	To this aim, consider the action of  $G$ on the $k$-fold product $\g^k$ which is  the adjoint action in each factor. Then the orbit of $(H_1, \ldots, H_k) \in \t_\Theta^k$ is the image of 
	\begin{equation}
		\label{eq-mergulho-produto}
		f_{(H_1, \ldots, H_k)}: 
		\F_\Theta \to \g^k, \qquad u b_\Theta \mapsto 
		(\Ad(u)H_1, \ldots, \Ad(u)H_k),
	\end{equation}
	and one readily sees that the  induced (semi-definite) metric is given by
	\begin{eqnarray}
		\langle X \cdot b_\Theta,\, Y \cdot b_\Theta \rangle_{b_\Theta} 
		& = & 
		\sum_\alpha (\alpha(H_1)^2 + \cdots + \alpha(H_k)^2) \left( X_\alpha, \,  Y_\alpha \right) .
		\label{eq-metrica-mergulho-produto}
	\end{eqnarray}
	Thus, we arrive at the more general map
	\begin{equation}
		\label{eq-metrica-mergulho-produto-map}
		\mu^{(k)}: \t_\Theta^k \to \overline{\mathcal{M}}
		\qquad
		(H_1, \ldots, H_k) \mapsto 
		(\alpha(H_1)^2 + \cdots + \alpha(H_k)^2)_{[\alpha]}.
	\end{equation}
	
	We further notice that the isotropy of $G_{(H_1,...,H_k)}$ is the intersection of the isotropies $G_{H_i}$. Therefore, 
	\[ G_\T \subseteq G_{(H_1,...,H_k)} \quad \Longleftrightarrow\quad  G_\T \subseteq G_{H_i},\quad i=1,...,k,\]
	thus justifying our special interest in $\t_\Theta^k$ instead of seemingly more general orbits in $\g^k$.
	Our next step is to study the properties of $\mu^{(k)}$ through its  geometry.

	\section{Metrics induced by products of the adjoint representation}
	
	Given $H\in\t_\T$, observe that 
	\[\mu^{(k)}(H,0,...,0)=\mu^{(k)}(0,H,0,...,0)=...=\mu^{(k)}(0,...,0,H).\] 
	On the other hand, the shifting $(H,0,...,0)\mapsto (0,H,0,...,0)$ is realizable by applying an orthogonal transformation to the lines of the matrix $(H|0|...|0)$, where $H$ and $0$ are seem as column vectors. Our first aim is to show how this invariance holds for every orthogonal transformation and use it to derive properties of the maps $\mu^{(k)}$.

	Denote the image of $\mu^{(k)}$ by $\ov{\mathcal{M}}^{(k)}$ and
	let $r = \dim \t_\T$ be the rank of  $\F_\T$. We now choose a basis for $\g$ in order to identify $\g^k$ with $r\times k$ matrices.
%
%
%
	Label the subset of simple roots outside of $\T$ by $\{\alpha_1,...,\alpha_r\} = \Sigma-\T$ and complete it with a choice of representatives $\{\alpha_1,...,\alpha_n\}\subseteq \Pi^+-\lr\T $, where $n\geq r$ is the number of isotropy summands.
	Denote $\m_i = \m_{[\alpha_i]}$, so that
	$$
	\m = \m_1 \oplus \cdots \oplus \m_n,
	$$
	is the decomposition of $\m$ into isotropy summands. 
	Note that the restriction of the linear functionals $\alpha_1,...,\alpha_r$ to $\t_\T$ form a basis of $\t_\T^*$.
	We endow $\t_\T$ with the inner product
	\begin{equation}
	\label{eq:basedual}
		\lr{H,H'} = \sum_{i,j=1}^r \alpha_i(H)\alpha_j(H'),
	\end{equation}
	and fix the orthonormal basis $\omega_1,...,\omega_r\in\t_\T$, dual to the restriction of  $\alpha_1,...,\alpha_r$ to $\t_\T$.
	From now on, we identify $\t^k_\T$ with the space of $r\times k$ matrices $\gl(r,k)$ by sending $(H_1,...,H_k)$  to the matrix $X$ whose $(i,j)$ entry is the $i$-th coordinate of $H_j$ in this fixed basis.
	
	One notices that $\mu^{(k)}$ becomes
	\begin{equation}
		X \mapsto 
		(x_1, \ldots, x_n) \in \R^n,
	\end{equation}
	where 
	\begin{equation}
		\label{metric-map-coordinates}
		x_i = \alpha_i(H_1)^2 + \cdots + \alpha_i(H_k)^2 \in \R^n.
	\end{equation}
	
	Let  
	\[\mathrm{P}(k)\subseteq \mathrm{PS}(k)\subseteq \mathrm{S}(k)\subseteq
	\gl(k) = \gl(k,k)\]
	denote the subset of positive, positive semi-definite, symmetric and square matrices, respectively.
	The main result of this section is Proposition \ref{prop:transposta} below.
	
	\begin{proposicao}
		\label{prop:transposta}
		The induced metric map $\mu^{(k)}(X) = (x_1,...,x_n)$ factors through the Gram map
		$$
		\gamma^{(k)}: \gl(r,k) \to \mathrm{PS}(r),
		\qquad
		X \mapsto XX^T	
		$$
		where $X^T$ is the transpose of $X$. Moreover, 
		each $x_i$ is a linear function on the entries of $XX^T$.
	\end{proposicao}
	
		Note that if we denote by $l_1,...,l_r$ the lines of $X$, as vectors in $\R^k$, then $\gamma^{(k)}(X)=XX^T$ is the \textit{Gram} matrix of $\{l_1,...,l_r\}$. That is, the $(i,j)$-th entry of $XX^T$ is $\lr{l_i,l_j}$.
	
		\begin{proof41}
		Write $H_i=(h_{1,i},...,h_{r,i})$. The lines of $X=(H_1,...,H_k)$ are the vectors $l_i=(h_{i,1},...,h_{i,k})$. 
		Since each root is a linear combination of simple roots and  roots in $\T$ vanish when restricted to $\t_\T$, it follows that there are positive integers $a_{1s},...,a_{rs}$ such that
		$$
		[\alpha_{r+s}] = 
		a_{1s} [\alpha_1] +...+a_{rs} [\alpha_r]
		$$
		for every $s \leq n - r$.
		We conclude that
		\begin{align}
			\label{structure_constants}\alpha_s(H_i) &= h_{s,i}, \qquad & \text{ for } s\leq r,\\\nonumber
			\alpha_{r+s}(H_i) &= a_{1s} h_{1,i} + \cdots + a_{rs} h_{r,i},\qquad &\text{ for } 0< s\leq n- r .
		\end{align}
		Therefore,
		\begin{eqnarray}
			x_1  &=& |l_1|^2, \qquad  \cdots \qquad x_s  =|l_s|^2 = \lr{l_s,l_s}, \qquad  \label{mus}\\
			& \vdots& \nonumber\\
			x_r  &=& |l_r|^2\nonumber\\
			x_{r+1} & = &   |a_{11}l_1+...+a_{r1}l_r|^2 \nonumber\\
			&\vdots& \nonumber\\
			x_{r+s} & =&    |a_{1s}l_1+...+a_{rs}l_r|^2 = \sum_{ij} a_{is}a_{js} \lr{l_i,l_j}
			\label{murs}\\
			&\vdots& \nonumber\\
			x_n & = &   |a_{1n}l_1+...+a_{rn}l_r|^2 \nonumber
		\end{eqnarray}
		where $|\cdot|$ stands for the standard euclidean norm in $\R^k$.
		The proposition follows since $x_i$ is a linear function on $\lr{l_i,l_j}=\sum_q h_{i,q}h_{j,q}$, the  entries of the matrix $XX^T$.
	\end{proof41}
	
If follows that $\mu^{(k)}$ factors through $\gamma^{(k)}$ and a linear map $\mu$: 
	$$
	\begin{tikzcd}
		\gl(r,k) \arrow[swap]{dr}{\gamma^{(k)}} \arrow{rr}{\mu^{(k)}} & & \overline{\mathcal{M}} \\
		& \mathrm{PS}(r) \arrow[swap]{ur}{\mu}.
	\end{tikzcd}
	$$
To study $\gamma^{(k)}$ we now focus our attention on the Gram map $\gamma = \gamma^{(r)}$
		$$
		\gamma: \gl(r) \to \mathrm{PS}(r),
		\qquad
		X \mapsto XX^T	
		$$
Given $k \leq r$, we consider $\gl(r,k)$ included in $\gl(r)$ as the leftmost $k$ columns, with zero elsewhere, so that $$\gamma^{(k)} = \gamma|_{\gl(r,k)}$$
	
	Note that ${\mathrm Gl}(r)$ acts on $\gl(r)$ both by right and left multiplication, and both actions preserve rank.  Given $V\subseteq \gl(r)$, denote the subset of rank $k$ matrices in $V$ as $V_k$.
	
	\begin{lema}
		\label{lem-gram} Let $X\in \gl(r)$.
		\begin{enumerate}[(i)]
			\item If $u \in \mathrm O(r)$, then 
			$\gamma(Xu) = \gamma(X)$ and $\gamma(uX) = u\gamma(X)u^{-1}
			$;
			
			\item If $\gamma(X)=\gamma(Y)$, $Y \in \gl(r)$, then $X = Y u$ for some $u \in {\mathrm O}(r)$. If $X,\, Y \in {\mathrm Gl}(r)$, then this $u$ is unique; 
			
			\item The symmetric square root $\sigma: \mathrm{PS}(r) \to \mathrm{S}(r)$, $Y \mapsto \sqrt{Y}$, is a continuous section of $\gamma: \gl(r) \to \mathrm{PS}(r)$;
			
			
			\item If $X \in \gl(r)_k$, then $Xu \in \gl(r,k)$ for some $u \in {\mathrm O}(r)$;
			
			\item 
			$\gamma( \gl(r)_k ) = \gamma( \gl(r,k)_k ) = \mathrm{PS}(r)_k$;
			
			\item For $0<k<r$, there does not exist a continuous section of the surjective restriction $\gamma^{(k)}: \gl(r,k)_k \to \mathrm{PS}(r)_k$;
			
			\item $\mathrm{PS}(r)$ is the convex closure of $\gamma( \gl(r,1)_1 )$.
		\end{enumerate}
	\end{lema}
	
	\begin{proof}
		Item $(i)$ is immediate from the definition, since $u^T = u^{-1}$ for $u \in \O(r)$.
		Item $(ii)$ follows from the polar decomposition $X = Pu$, where $P$ is  positive semi-definite and $u \in {\mathrm O}(r)$.
		For a proof of item $(iii)$, we refer to \cite{bouldin}.
		

		For item $(iv)$, we first observe that, by right multiplying by a permutation matrix, which lies in ${\mathrm O}(r)$, we can suppose w.l.o.g.\ that the first $k$ columns $v_1, \ldots v_k$ of $X$ are linearly idenpendent.  The remaining $r-k$ columuns are then given by
		$a_{11} v_1 + \ldots + a_{k 1} v_k$, $\ldots,$ 
		$a_{1,n-k} v_1 + \ldots + a_{k,n-k} v_k$.
		Then $A = (a_{ij})$ is a rectangular $k \times (r-k)$ matrix, and we consider the upper triangular invertible matrix $g$ with two identity diagonal blocks of sizes $k \times k$ and $(r-k) \times (r-k)$, with $A$ above them and zero below. Then
		$X = Y g$, where $Y \in \gl(r,k)$ with first $k$ columns given by $v_1, \ldots v_k$ and remaining $r-k$ columuns given by zero.
		By applying the Gram-Schimdt process (seen as a QR decomposition), we get $g$ such that $g^T = u V$, with $u \in {\mathrm O}(r)$ and $V$ upper triangular.  Thus $g = V^T u^{-1}$, with $V^T$ lower triangular. Moreover,
		$$
		X u = Y V^T \in \gl(r,k)
		$$
		since, for $i > k$, $v^T$ maps $e_i$ to the subspace spanned by $e_i, e_{i+1}, \ldots, e_r$, which is annihilated by $Y$.
		
		Item $(v)$ follows by first observing that  items $(i)$ and $(iv)$ together imply that $\gamma( \gl(r)_k ) = \gamma( \gl(r,k)_k )$.  To show the equality of the last with $\mathrm{PS}(r)_k$, it is sufficient to note that  both $\gamma$ and $\sigma$ preserve rank. This is clear for $\sigma$ since the rank is the number of nonzero eigenvalues. For $\gamma$, this follows from 
		$\dim \mathrm{Im}(XX^T) = \dim \mathrm{Im}(X)$, which holds since $\mathrm{Im}(XX^T) \subseteq  \mathrm{Im}(X)$ and since $XX^T$ contains as a rank $k$ submatrix the Gram matrix of $k$ l.i.\ lines of $X$.
		
		For item $(vi)$, let $H$ be the matrix with one identity diagonal block of size $k \times k$, one zero diagonal block of size $(r-k) \times (r-k)$, and zero elsewhere.  Note that $H \neq 0$ is not the identity matrix, since $0 < k < r$, furthermore $H = H^T = H^2$, so that $H$ is a projection matrix. 
		Note that the orbit $\O(r)H \subset \gl(r,k)_k$ is given by the first $k$ orthogonal columns of $\O(r)$, and zero elsewhere, thus
		$$
		\O(r)H = \mathrm{St}_k(r)
		$$
		the Stiefel manifold of orthogonal $k$-frames of $\R^r$. Denote conjugation by $\Ad(g)H = g H g^{-1}$. Since $\gamma(H) = H$, by item (i), $\gamma^{(k)}$ restricted to $\mathrm{St}_k(r)$ is the standard projection to
		$$
		\Ad(\O(r)) H = \mathrm{Gr}_k(r)
		$$
		the projection representation of the Grassman manifold of $k$-dimensional subspaces of $\R^r$.
		Furthermore, we have that $(\gamma^{(k)})^{-1}(H)$ are the matrices of $\gl(r,k)$ whose first $k$ lines are orthonormal, and the remaining lines are orthogonal to them.  Since such lines belong to $\R^k$, it follows that 
		$$
		(\gamma^{(k)})^{-1}(H) = \O(k)H
		$$
		where $\O(k)$ is the subgroup of $\O(r)$ with one orthogonal block of size $k \times k$, one identity diagonal block of size $(r-k) \times (r-k)$, and zero elsewhere.
		Now let $X \in \gl(r,k)$ be such that $\gamma^{(k)}(X) = uHu^{-1}$ for some $u \in \O(r)$. Then $\gamma^{(k)}(u^{-1} X) = H$, so that $u^{-1} X \in  \mathrm{St}_k(r)$, and then $X \in \mathrm{St}_k(r)$.  It follows that
		$$
		(\gamma^{(k)})^{-1}(\mathrm{Gr}_k(r)) = \mathrm{St}_k(r)
		$$
		Thus, a continuous section of $\gamma^{(k)}$ would restrict to a continuous section of the principal $\O(r)$-bundle $\mathrm{St}_k(r) \to \mathrm{Gr}_k(r)$ which would render it a trivial bundle. A contradiction, since it is well known (for example, by computing their fundamental group) that this bundle is not trivial, for $0<k<r$.
		
		
		For item (vii), it is well known that $\mathrm{PS}(r)$ is the convex closure of its rank $1$ matrices $\mathrm{PS}(r)_1$ (see, for example, Proposition 1.21 of \cite{cone}). The result then follows from item (v), since $\mathrm{PS}(r)_1 = \gamma( \gl(r,1) )$.
	\end{proof}
	
	Note that $\mathrm{P}(r) = \mathrm{PS}(r)_r$ and that the boundary of $\mathrm{P}(r)$ in $\mathrm{S}(r)$ is the union of the $\mathrm{PS}(r)_k$, $k< r$.  Items $(ii)$ and $(iii)$ imply the well known fact that the restriction $\gamma: \mathrm{Gl}(r) \to \mathrm{P}(r)$ is a trivial principal ${\mathrm O}(r)$-bundle.
	Items $(iv)$ and $(v)$ suggests that we could use the right ${\mathrm O}(r)$-action to gauge a section of $\gamma$ to get a section of $\gamma^{(k)}$, nevertheless item $(vi)$ shows that this cannot be done continuously.
	This last assertion will be used in the next sections.

Theorem \ref{ithm:convex} now follows by putting together the linearity of $\mu$ (Proposition \ref{prop:transposta}) and the geometry of $\gamma$ (Lemma \ref{lem-gram}).
	



\begin{proof12}
Item $(vii)$ of Lemma \ref{lem-gram} guarantees that $\rm{PS}(r)$ is the convex closure of $\gamma^{(1)}(\gl(r,1))$. Since $\mu$ is linear, one concludes that $\ov{\mathcal M}^{(r)}=\rm{conv}\ov{\mathcal M}^{(1)}$.

	Finally, if the $T$-roots $\F_\T$ are of Type A then $\mu$ is a linear isomorphism. Indeed, the number of summands is $n=r(r+1)/2$, which coincides with the dimension of the space of symmetric $r \times r$ matrices. 
	Furthermore, the representatives of $\Pi_\T$ besides $[\alpha_1], \ldots, [\alpha_r]$ can be relabeled as $[\alpha_{pq}]$ such that
	$$
	[\alpha_{pq}] = [\alpha_p] + [\alpha_{p+1}] + \cdots + [\alpha_{q-1}] + [\alpha_{q}]
	\quad(1\leq p < q \leq k)
	$$
	so that
	$$
	\alpha_p(H_i) = h_{p,i},
	\qquad\qquad
	\alpha_{pq}(H_i) = h_{p,i} + \cdots + h_{q,i},
	$$
	with corresponding functionals
	\begin{equation}
		\label{mudef}
		x_{p}(Y) = y_{pp}, \qquad\qquad
		x_{pq}(Y) = \sum_{p\leq i,j \leq q} y_{ij}.
	\end{equation}
	Since $x_{pq}(Y)$ consists of the sum of $ y_{pq}$ with a linear combination of $y_{ij}$,   $p\leq i,j \leq q$, $(i,j)\neq(p,q)$, it follows that the set of functionals $\{x_p,x_{pq}\}$ are linearly independent. Hence, $\overline{\mathcal{M}}^{(r)}$ has non-empty interior in $\overline{\mathcal{M}}$, as desired.
		\end{proof12}

	\vspace{12pt}
	
	If follows that one does not gain new metrics by increasing the number of copies in the product of the adjoint representation. That is, the best we can get  is by taking $r$ products, where $r$ is the rank of the flag manifold. 
	
	Theorem \ref{ithm:convergence} shows that $\ov{\mathcal M}^{(r)}$ do not realize every metric in $\F_\T$, even in the type $A$ case. In this case, however, since the interior of $\ov{\mathcal M}^{(r)}$ is non-empty, one concludes that every invariant metric can be realized by a (non-convex) linear combination of metrics in $\ov{\mathcal M}^{(r)}$. Therefore, every metric can be realized as an orbit in the non-Euclidean representation $(\g^r,(\cdot,\cdot)){\times} (\g^r,-(\cdot,\cdot))$ given by the adjoint in each of the $2r$ coordinates. We do not extent any longer in this direction, since we are mostly interested in collapsing phenomena, and the non-Euclidean context conceal such a phenomena with manifolds with degenerated norms.

	\section{Realizing the projected Ricci flow
	\label{sec:realizing}
	}
	
	The image of the map $\mu^{(1)}: \t_\Theta \to \overline{\mathcal{M}}$ of metrics induced by an adjoint orbit 
	is a cone whose dimension is the rank $r$ of the flag manifold $\F_\T$, embedded into an eucliden space whose dimension is the number of its isotropy summands. For flag manifolds with two or three summands, it follows then from Table \ref{tab-upto3} that the set of invariant metrics that can be realized as an adjoint orbit is thin on the simplex $\mathcal{T}$: $x+y+z=1$, the phase space of the projected Ricci flow.
	
	This can be improved by taking products and the best we can get is the map $\mu^{(r)}: \t_\Theta^r \to \overline{\mathcal{M}}$,  
	where
	$$
	\dim \t_\Theta^r = \overbrace{r + \cdots + r}^r = r^2
	$$
	Thus, this improves the situation only for the case of Type II, which has three summands and is the only one with rank $r=2 > 1$.
	
	From now on we assume that $\F_\T$ is of Type II, which consists of two infinite families of types $A$ and $D$ and a flag manifold of type $E$ (see Table \ref{tab-typeII}), note that these correspond to simply laced Lie algebras. Note futhermore that their $T$-roots are of Type A (see Table \ref{tab-upto3}). 
	
	\begin{table}[h!]
		\caption{ \label{tab-typeII} Type II flag manifolds: type $A$, $D$, $E$ (see \cite{kimura})}
		\hspace{-1cm}
		\begin{tabular}{lllllll}
			Flag Manifold   & & $d_1$ & & $d_2$ & & $d_3$ \\ \hline \hline
			$SU(m+n+p)/S(U(m)\times U(n) \times U(p))$ && $2mn$ &&$2mp$ && $2np$ \\ \hline
			$SO(2\ell)/U(1)\times U(\ell-1)$, $\ell \geq 4$ && $2(\ell-1)$ && $2(\ell -1)$ && $(\ell-1)(\ell-2)$ \\ \hline
			$E_6/SO(8)\times U(1)\times U(1)$ && 16 && 16 && 16
		\end{tabular}
	\end{table}
	
	By Theorem \ref{ithm:convex}, the image $\overline{\mathcal{M}}^{(2)}$ of $\mu^{(2)}$
	has non-empty interior in $\overline{\mathcal{M}}$, we now characterize this image.
	First, we characterize the images of $\gamma^{(2)}$ and $\gamma^{(1)}$.
	The image of $\gamma^{(2)}$ is $\mathrm{PS}(2)$, where $Y = \begin{smallpmatrix} x & z \\ z & y \end{smallpmatrix} \in \mathrm{S}(2)$ belongs to $\mathrm{PS}(2)$ iff
	$\det Y = xy - z^2 \geq 0$ and $\mathrm{tr\,} Y = x+y \geq 0$. Thus, $\mathrm{PS}(2)$ is the intersection of the half-space $x+y \geq 0$ with the convex closure of the cone surface
	\begin{eqnarray}
		\label{eq-cone0}
		z^2 = xy
	\end{eqnarray}
	whose quadratic form has matrix $Q = \begin{smallpmatrix}
		\,\,\,0 & -1/2 & \,\,\,0 \\ 
		\,\,\,0 & \,\,\,0 & \,\,\,0 \\ 
		-1/2 & \,\,\,0 & \,\,\,1 \\ 
	\end{smallpmatrix}$.
	The boundary of this convex cone is made up of the rank 1 positive-semidefinite matrices, hence it is the image of $\gamma^{(1)}$.
	
	Since the $T$-roots of Type II are characterized by three roots $\alpha_1,\alpha_2,\alpha_3$ satisfying 
	\[
	[\alpha_1]+[\alpha_2]=[\alpha_3]
	\]
	It follows from \eqref{mus} and \eqref{murs} that
	$$
	x_1 = |l_1|^2 \qquad x_2 = |l_2|^2 \qquad x_3 = |l_1 + l_2|^2 
	= |l_1|^2 + |l_2|^2 + 2\lr{l_1,l_2}
	$$
	so that 
	$
	\mu\begin{smallpmatrix} x & z \\ z & y \end{smallpmatrix}
	= (x,y,x+y+2z)
	$,
	whose matrix is $M = 
	\begin{smallpmatrix}
		1 & 0 & 0 \\ 
		0 & 1 & 0 \\ 
		1 & 1 & 2 \\ 
	\end{smallpmatrix}$, which can also be obtained by \eqref{mudef}.
	It sends the cone surface \eqref{eq-cone0} to its pullback by $\mu^{-1}$, given by the cone surface in $\R^3$
	\begin{eqnarray}
		\label{eq-cone1}
		x^2 + y^2 + z^2 = 2(xy + xz + yz)
	\end{eqnarray}
	whose quadratic form has matrix $M^{-T} Q M^{-1} = \begin{smallpmatrix}
		\,\,\,\,\,1 & -1 & -1 \\ 
		-1 & \,\,\,\,\,1 & -1 \\ 
		-1 & -1 & \,\,\,\,\,1 \\ 
	\end{smallpmatrix}$.
	Its negative eigenspace is normal to the simplex $\mathcal{T}$ and positive eigenspace is spanned by this simplex, thus, it corresponds to a right circular cone surface. Denote by $\mathcal{C}$ the circle given by its intersection with the simplex.
	\begin{center}
		\includegraphics[width=7cm]{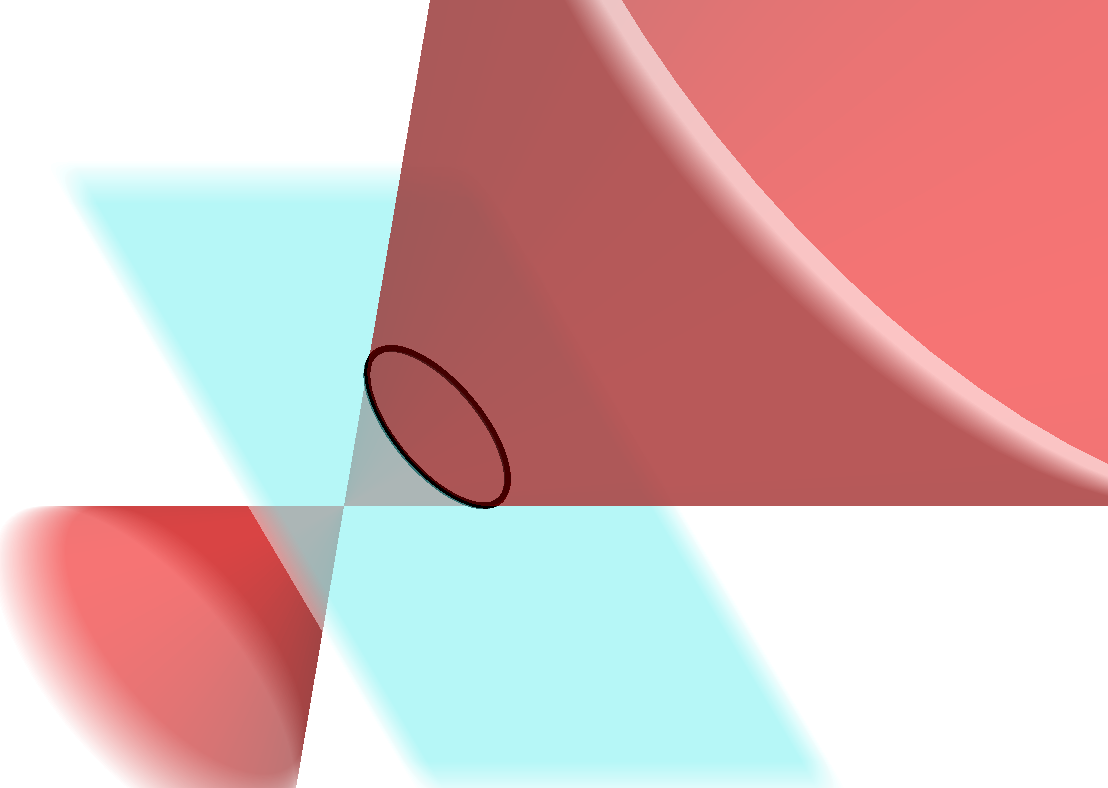}
	\end{center}
	It follows from the previous paragraph that the image of $\mu^{(1)} = \mu \circ \gamma^{(1)}$ is the cone surface \eqref{eq-cone1} on the first orthant and the image of $\mu^{(2)} = \mu \circ \gamma^{(2)}$ is the convex closure of this cone surface. This gives the first part of the following.
	
	\begin{proposicao}
		The semi-definite metrics $\overline{\mathcal{M}}^{(2)}$ realized by products of adjoint orbits is the convex closure of the right circular cone generated by $\mathcal{C}$ with vertex at the origin.
		
		Furthermore, $\overline{\mathcal{M}}^{(2)}$ is positively invariant by the unprojected Ricci flow.
	\end{proposicao}
	\begin{proof}
		To prove the second part we show that the unprojected Ricci flow with vector vector field $R(x,y,z)$ points inwards at the boundary. 
		By $\eqref{eq-cone1}$, we have that $\overline{\mathcal{M}}^{(2)}$ is given in the first orthant by $F(x,y,z) \leq 0$, where
		$$
		F(x,y,z) = x^2 + y^2 + z^2 - 2(xy + xz + yz)
		$$
		Thus its outward pointing normal is
		$$
		\nabla F(x,y,z) = -2(-x + y + z, \,x - y + z, \, x + y - z)
		$$
		and we must show that the inner product $R(x,y,z) \cdot \nabla F(x,y,z) \leq 0$ at the boundary $F(x,y,z)=0$. For this, we will use the algebraic identities
		\begin{eqnarray}
			\label{eq-soma}
			(-x + y + z) + (x - y + z) + (x + y - z) = x + y + z \\
			\label{eq-produto2}
			(x - y + z)(x + y - z) = x^2 - (y-z)^2 =  x^2 - y^2 - z^2 + 2yz
		\end{eqnarray}
		and the following identities that follow from $F = 0$
		\begin{eqnarray}
			\label{eq-quadrados}
			x^2+y^2+z^2 & = & 2(xy + xz + yz) \\ 
			\label{eq-produto3}
			q := (-x + y + z)(x - y + z)(x + y - z) & = & -8xyz
		\end{eqnarray}
		To obtain the latter, first note that by \eqref{eq-produto3} we have $-y^2-z^2 = x^2 - 2(xy + xz + yz)$, replacing in \eqref{eq-produto2} we get 
		\begin{eqnarray}
			(x - y + z)(x + y - z) & = & 2x(x-(y+z)) \nonumber
		\end{eqnarray}
		It follows that 
		\begin{eqnarray}
			q & = &  (-x+y+z)2x(x-(y+z)) \nonumber\\
			& = &-2x(x-(y+x))^2 \nonumber
		\end{eqnarray}
		and then \eqref{eq-produto3} follows from $(x-(y+z))^2 = F(x,y,z) + 4yz$.
		
		In what follows we use the Ricci vector fields for three summands Type II flag manifolds of types $A$, $D$ and $E$ given in \cite{ricci}.  Type $E$ has the same Ricci vector field as type $A$ with $m=n=p=1$, thus we need only to consider types $A$ and $D$.

		\begin{itemize}
			\item Type $A$: the Ricci vector field $R(x,y,z)$ equals
			$$
			\hspace{-3cm}
			(-x(p(x^2-(y-z)^2) + 2(m+n)yz),\,
			-y(n(y^2-(z-x)^2) + 2(m+l)xz),\,
			-z(m(z^2-(x-y)^2) + 2(n+l)xy) )
			$$
			where $m,n,p$ are positive integers.
			Using \eqref{eq-produto2} and \eqref{eq-produto3} 
			we have that $R(x,y,z) \cdot \nabla F(x,y,z)$ equals
			\begin{eqnarray*}
				&& 2x(p(x^2-(y-z)^2) + 2(m+n)yz)(-x+y+z)\\
				&+& 2y(n(y^2-(z-x)^2) + 2(m+l)xz)(+x-y+z)\\
				&+& 2z(m(z^2-(x-y)^2) + 2(n+l)xy)(+x+y-z) \\
				&=& 2(px + ny + mz )(-x + y + z)(x - y + z)(x + y - z) \\
				&+& 4xyz ( (m+n)(-x+y+z) + (m+l)(+x-y+z) + (n+l)(+x+y-z)) \\ 
				&=& (4xyz - 8xyz) ( 2px + 2ny + 2mz ) \\
				&=& -8 xyz( px + ny + mz ) \leq 0
			\end{eqnarray*}
			
			\item Type $D$ the Ricci vector field $R(x,y,z)$ equals
			$$
			\hspace{-3cm}
			(-x((\ell-2)(-x^2-(y-z)^2) + 2\ell yz),\,
			-y((\ell-2)(y^2-(z-x)^2) +2\ell xz),\,
			-z(2(z^2-(x-y)^2) + 4(\ell-2)xy) )
			$$
			where $\ell \geq 4$ is an integer.
			Using \eqref{eq-soma}, \eqref{eq-produto2}, \eqref{eq-produto3} we have that $R(x,y,z) \cdot \nabla F(x,y,z)$ equals
			\begin{eqnarray*}
				&& (2x(\ell-2)(+x - y + z)(+x + y - z)+ \phantom{(\ell-2} 4\ell xyz)(-x+y+z)  \\
				&+& (2y(\ell-2)(-x + y + z)(+x + y - z) + \phantom{(\ell-2} 4\ell xyz)(+x-y+z) \\ 
				&+& \phantom{(\ell-2)} (4z(+x - y + z)(+x - y + z) + 8(\ell-2) xyz)(+x+y-z) \\ 
				&=&  2(\ell-2)q(x+y) -\ell q (z) \\
				&+&  \phantom{(\ell-2)}4qz - (\ell-2)q(+x+y-z)\\ 
				&=&  (\ell-2)q(x+y) + (-\ell + 4 + (\ell-2)) q(z) \\
				&=&  -8xyz((\ell-2)q(x+y) + 2z) \leq 0
			\end{eqnarray*}
		\end{itemize}
		Both are nonpositive and zero only at the coordinate planes $x$, $y$ or $z=0$.
	\end{proof}
	
	Now we consider the projected Ricci flow and realize it as a product of two adjoint orbits. Note that the convex cone $\overline{\mathcal{M}}^{(2)}$ intersects the simplex $\mathcal{T}$ in a disk $\mathcal{D}$ with boundary $\mathcal{C}$, illustrated in Figure \ref{JapanFlag}. 
	Furthermore, we have the following.
	
	\begin{proposicao}
		The disk $\mathcal{D}$ is positively invariant.
	\end{proposicao}
	\begin{proof}
		The projected Ricci flow $X(x,y,z)$ is of the form \eqref{eq-def-X}. Using the notation of the proof of the previous proposition, 
		since $F(x,y,z)$ is a homogeneous function of degree $2$, it follows that $\nabla F(x,y,z) \cdot (x,y,z) = 2 F(x,y,z)=0$, whenever $F(x,y,z) = 0$. Thus $\nabla F(x,y,z) \cdot X(x,y,z)$ is a positive multiple of $\nabla F(x,y,z) \cdot R(x,y,z)$ on the cone boundary $F(x,y,z)=0$. The result then follows from the previous proof.
	\end{proof}
	
	
	
	
	
	\begin{teorema}\label{thm:embedding}
		The projected Ricci flow $g_t$ in the disk $\mathcal{D}$ is realizable as an equivariant embedding of the flag manifold.
	\end{teorema}
	
	
	Now let $\sigma$ be the continuous section of $\gamma$ given by Lemma \ref{lem-gram}, item $(iii)$. It follows that 
	$\gamma^{(r)}(\sigma(Y)) = \sqrt{Y}\sqrt{Y}^T = \sqrt{Y}^2 = Y$, so that $\sigma$ is a continuous section of $\gamma^{(r)}$ which is smooth in the interior $\mathrm{P}(r)$ of $\mathrm{PS}(r)$. 
	If $\mu$ is injective, let $\mu^{-1}$ be its inverse when restricted its image, then
	$$
	\tau = \sigma \circ \mu^{-1}: \overline{\mathcal{M}}^{(r)} \to \mathrm{PS}(r) \subseteq \gl(r)
	$$
	is a continuous section of $\mu^{(r)}$, which is smooth in the interior of $\overline{\mathcal{M}}^{(r)}$.
	$$
	\begin{tikzcd}
		\gl(r) \arrow[swap]{dr}{\gamma^{(r)}} \arrow{rr}{\mu^{(r)}} & & \overline{\mathcal{M}}^{(r)} 
		\arrow[ll, bend right, swap, "\tau"]
		\\
		& \mathrm{PS}(r) \arrow[swap]{ur}{\mu}
		\arrow[ul, bend left=50, "\sigma"]
	\end{tikzcd}
	$$
	Indeed,
	$$
	\mu^{(r)} \circ \tau = (\mu \circ \gamma^{(r)}) \circ (\sigma \circ \mu^{-1}) = \mu \circ \mu^{-1} = \mathrm{id}
	$$
	Note that $\sigma$, and consequently $\tau$, can be gauged by the right $\mathrm{O}(r)$ action on $\gl(r)$.
	
	To produce an embedding of ${\mathcal D}$, consider the topological 2-disk 
	$$
	\widetilde{\mathcal{D}} = \tau(\mathcal{D}) \subset \gl(2) = \t_\T^2
	$$
	where the last identification sends the $i$-th column $(h_{1i}, h_{2i})$, $i=1,2$, to the $i$-th coordinate $h_{1i} \omega_1 + h_{2i} \omega_1 \in \t_\T$. Note that, since $\mu^{-1}(x,y,z) = \begin{smallpmatrix}
		x & \frac{1}{2}(z-x-y)\\
		\frac{1}{2}(z-x-y) & y 
	\end{smallpmatrix}
	$, we have that
	$$
	\tau(x,y,z) = \sqrt{
		\begin{pmatrix}
			x & \frac{1}{2}(z-x-y)\\
			\frac{1}{2}(z-x-y) & y 
		\end{pmatrix}
	}
	$$
	
	Then $g_t$ in $\mathcal{D}$ is conjugate to the flow $H_t = \tau(g_t)$ in $\widetilde{\mathcal{D}}$. The latter is  continuous on both $H$ and $t$, and is such that $(\F_\T, g_t)$ is isometric to the adjoint orbit
	$$
	\Ad(G)H_t \subseteq \g^2
	$$
	with the induced metric, for all appropriate $t$. 
	From Theorem \ref{ithm:convergence}, it follows that, when $g_t$ converges to a degenerate metric in the boundary $\mathcal{C}$ of $\mathcal{D}$, the adjoint orbits of the family $H_t$ realize the Gromov-Hausdorff collapse of the family of metric spaces defined by $(\F_\T, g_t)$ as a Hausdorff limit of sets inside $\g^2$.  
	
	This collapse has already been obtained intrinsically (Lemma 4.5 of \cite{ricci}), but the point of the present result is to realize them isometrically in Euclidean space.


	    

	\vspace{12pt}
	\begin{exemplo}
		Let $\F_\T$ be a three summand flag manifold of $SU(m)$. Let us show how some collapses are explicitly realized as Grassmann manifolds. First, suppose that $g_t$ collapses to the degenerate metric $L = (\frac{1}{2}, \frac{1}{2}, 0)$ and
		observe that
		$$
		\tau(L) = \sqrt{
			\begin{pmatrix}
				\frac{1}{2} & -\frac{1}{2}\\
				-\frac{1}{2} & \frac{1}{2} 
			\end{pmatrix}
		}
		=
		\begin{pmatrix}
			\frac{1}{2} & -\frac{1}{2}\\
			-\frac{1}{2} & \frac{1}{2} 
		\end{pmatrix}
		=
		(H, -H)
		$$
		since this matrix is a projection, with first column
		$$
		H = \frac{1}{2} \omega_1 - \frac{1}{2} \omega_2
		$$
		where $\omega_1,\omega_2\in\t_\T$ is the basis dual to $\alpha_1,\alpha_2$, chosen just after (\ref{eq:basedual}). Note that the isotropy of $(H,-H)$ of the adjoint action in $\g^2$ is the same as the isotropy of $H$ under the adjoint action in $\g$.
		Note also that 
		$$
		\alpha_1(H)=\frac{1}{2}, \qquad 
		\alpha_3(H)=\alpha_1(H) + \alpha_2(H) = \frac{1}{2} - \frac{1}{2} = 0.
		$$
		Thus concluding  that $H$ has the isotropy of a Grassmanian manifold.  
		Note that by multiplying $\tau(L)$ on the right by the rotation matrix of 45$^\circ$, we get
		$
		\begin{pmatrix}
			\frac{\sqrt{2}}{2} & 0 \\
			-\frac{\sqrt{2}}{2} & 0 
		\end{pmatrix}
		= (\sqrt{2}H, 0)
		$,
		which makes explicit the fact that $g_t$ converges to an adjoint orbit in $\g$.
		
		Likewise, the points  $M = (\frac{1}{2}, 0, \frac{1}{2})$ and $K = (0, \frac{1}{2}, \frac{1}{2})$ are entirely analogous. For instance,
		$$
		\tau(M) = \sqrt{
			\begin{pmatrix}
				\frac{1}{2} & 0\\
				0 & 0 
			\end{pmatrix}
		}
		=
		\begin{pmatrix}
			\frac{\sqrt{2}}{2} & 0\\
			0 & 0 
		\end{pmatrix}
		=
		(H, 0)
		$$
		where $H = \frac{\sqrt{2}}{2} \omega_1$ is such that
		$$
		\alpha_1(H) = \frac{\sqrt{2}}{2} \qquad 
		\alpha_2(H) = 0.
		$$
		Therefore,  $H$ has the isotropy of a Grassman manifold.
		
		By Theorem \ref{ithm:convergence}, in all three cases $(\F_\T, g_t) = \Ad(G)H_t$ converges to the Grassmanian $\Ad(G)H$ in Hausdorff distance in $\g^2$, as $t \to \infty$.
	\end{exemplo}
	\vspace{12pt}

	Back to general case, the next result shows that we cannot expect that the whole boundary of $\mathcal{D}$ gets mapped to the same factor $\t_\T$ of $\t_\T^2$.  Indeed, let $\rho$ an arbitrary section of $\mu^{(2)} = \mu \circ \gamma^{(2)}$. Note that at the boundary, we have that
	$$
	\rho(\mathcal{C}) \subseteq \gl(2)_1
	$$
	since
	$$
	\mathcal{C} \subseteq \overline{\mathcal{M}}^{(1)} = \mu( \gamma^{(1)}( \gl(2)_1) ) =  \mu^{(2)}|_{\gl(2)_1}
	$$
	Even though $\gamma(\gl(2)_1) = \gamma(\gl(2,1)_1)$, there is no way to continuosly gauge $\rho$ so that it sends $\mathcal{C}$ to $\gl(2,1)_1$.  More precisely, we have the following.
	
	\begin{proposicao}
		There does not exist a continuous section of $\mu^{(2)}$ which maps  $\mathcal{C}$ to $\gl(2,1)_1$.
	\end{proposicao}
	\begin{proof}
		Suppose that $\rho$ is such a section. Then for $(x,y,z) \in \mathcal{C}$ we have $\rho(x,y,z) \in \gl(2,1)_1$ so that
		\begin{equation}
			\label{secao-M2}
			(x,y,z) = \mu^{(2)}(\rho(x,y,z)) = \mu^{(1)}(\rho(x,y,z))
		\end{equation}
		Since $\mathcal{M}^{(1)}$ is the cone generated by $\mathcal{C}$ with vertex at the origin, we can define
		$$
		\widetilde{\rho}: \mathcal{M}^{(1)} \to \gl(2,1)_1
		\qquad
		\widetilde{\rho}(x,y,z) = \sqrt{(x+y+z)}\, \rho\!\left(\frac{x,\, y,\, z}{(x+y+z)}\right)
		$$
		where for $(x,y,z) \in \mathcal{M}^{(1)}$ we have 
		$x + y + z > 0$ and $(x,y,z)/(x+y+z) \in \mathcal{C}$.
		Since $\mu^{(1)}$ is homogeneous of degree $2$, equation \eqref{secao-M2} implies that
		$$
		\mu^{(1)}\widetilde{\rho}(x,y,z) = (x+y+z) \left(\frac{x,\, y,\, z}{(x+y+z)}\right) = (x,y,z)
		$$
		so that $\widetilde{\rho}$ is a continuous section of $\mu^{(1)}: \gl(2,1)_1 \to \mathcal{M}^{(1)}$.
		Since $\mu^{(1)} = \mu \circ \gamma^{(1)}$, where $\mu$ maps $ \mathrm{PS}(r)_1$ to $\mathcal{M}^{(1)}$, this corresponds to a continuous a section of $\gamma^{(1)}:  \gl(2,1)_1 \to \mathrm{PS}(r)_1$, which contradics item (vi) of Lemma \ref{lem-gram}.
	\end{proof}

	\section{Collapse and non-realization of metrics}\label{sec-collapse}

	This section aims at  explaining  why the vertex points $O,P,Q$ are not realized as orbit metrics. Indeed, there is no sequence of homogeneous  spaces whose limit realizes such metrics. As we will see, although $O,P,Q$ represents non-zero (degenerated) metrics,  Gromov--Hausdorff converging sequences to these points corresponds to sequences whose limits are single points (see Tables (27) and (31) in \cite{ricci}).	In order  to give a formal proof to this statement, we broaden our discussion to Gromov-Hausdorff limits of orbits. 
	
		Recall that the Gromov--Hausdorff distance between two metric spaces $(X_1,d_1)$ and $(X_2,d_2)$ is
	the infimum of the Hausdorff  distance between $X_1,X_2$ among all possible isometric embeddings of $X_1,X_2$. That is,  
	\[d_{GM}((X_1,d_1),(X_2,d_2))=\inf \{ d_H(X_1,X_2):(X_i,d_i)\subseteq (Z,d_Z)\text{ isometrically}\}, \]
	where  $d_H(X_1,X_2)$ denotes the Hausdorff distance between $X_1,X_2\subseteq Z$: 
	\[
	d_H(X_1,X_2)=\max\left\{ 
	\sup_{x_1 \in X_1} d_Z(x_1,X_2),\,  
	\sup_{x_2 \in X_2} d_Z(X_1,x_2)
	\right\}.
	\] 
	All limits of metric spaces below are in the Gromov--Hausdorff sense.
	Next, we recall Theorem 2.2 of \cite{ricci}, which characterizes limits of homogeneous metrics.
	
	As before, let $G/K$ be an homogeneous space, with $\g=\m\oplus\mathfrak k$, a $(\, ,)$-orthogonal decomposition. Identify $G$-invariant metrics in $G/K$  with the space of $K$-invariant inner products in $\m$. Let $g_n$ be a sequence of $G$-invariant metrics in $G/K$ converging to a (possibly) degenerate inner product $g$.
	Define
	\begin{align*}
		\ker g &=\{X\in \mathfrak m~:~g(X,\mathfrak m)=0\};\\
		\h &=  \text{ Lie algebra generated by }\ker g \oplus \k,
	\end{align*}
	which are $K$-invariant, since $g$, $\m$  and $\k$ are so. Let $H \leq G$ be the connected Lie subgroup with Lie algebra $\h$. 
	
	\begin{teorema}\label{thm:collapse}
		Let $(G/K,g_n)$, $g$ and $H$ be as above. Suppose that $H$ is closed, then
		\begin{equation*}
			\lim_{t\to \infty}(G/K, d_{g_n})=(G/H,d_F),
		\end{equation*}
		where $d_F$ is the distance metric induced by the (not necessarily smooth) $G$-invariant Finsler norm 
		\begin{align}\label{eq:Finsler-Lie}
			F_g: \mathfrak h^\perp&\to \mathbb R\\
			\nonumber	X &\mapsto \min\{ |\Ad(h)X + Z|_{g}:~ h\in H,~ Z \in \m \cap \h \}.
		\end{align}
		
	\end{teorema} 
	
	In particular, we have
	\begin{corolario}\label{cor:Finsler}
		if $g$ is $\Ad_G(H)$-invariant, then $F$  is the norm of the  Riemannian metric $g|_{\tilde{\n}}$, where $\tilde \n$ is the $g$-orthogonal complement of $\m\cap \h$ in $\m$.
	\end{corolario}
	
	For the remaining of the section, let $G\to O(V)$ be an orthogonal representation of $G$ in an Euclidean space $V$. 
	Given $v \in V$, recall the isotropy subgroup, given by
	$
	G_v = \{ g \in G:~ gv = v \}
	$,
	so that the representation induces a $G$-equivariant realization $i: G/G_v \to Gv \subseteq V$ with the induced metric.	More generally, one can consider $V$ a Riemannian manifold and  $G\times V\to V$ a smooth action by isometries. 
	
	Let $v_n$ be a sequence in $V$ converging to $v\in V$.
	One would expect that the Gromov--Hausdorff limit of a sequence of the orbits $Gv_n$ coincides with the limiting orbit $Gv$, together with its induced metrics from the embedding in $V$. Such a sequence  avoids the two relatively unexpected behaviours in Theorem \ref{thm:collapse}: the kernel of the limiting metric coincides with the isotropy of $v$, which is already a subalgebra; the limiting distance is induced by a Riemannian metric, not only a Finsler norm. This behaviour allows us to guarantee that many sequences of metrics, together with its limits, are not realizable as sequences of orbits. 


	
	\begin{teorema}\label{prop:convergence}
		Let $g_n$ be a sequence of $K$-invariant inner products in $\m$ converging to $g$,  a positive semi-definite bilinear form in $\m$.
		Then,
		\begin{enumerate}
			\item The Gromov-Hausdorff limit of $(G/K,d_{g_n})$ does not depend on $g_n$, only on the limit $g$, where $d_{g_n}$ is the geodesic distance induced by $g_n$;
			\item If $v_n \in V$ converges to $v \in V$, put $K_n = G_{v_n}$, $K = G_v$ and let $g_n$ be induced by the  natural embedding $i_n:G/K_n \to Gv_n$. Then $(G/K_n,d_{g_n})$ converges to $(G/K,d_g)$ in the Gromov-Hausdorff sense;
			\item Suppose that $\ker g\oplus\mathfrak{k}$ is not a subalgebra and let $K_n = G_{v_n}$. Then, there is no sequence of isometric equivariant embeddings $i_n:(G/K_n,g_n)\to (V,\tilde g)$ such that  $g_n\to g$.
		\end{enumerate}
	\end{teorema}

	\begin{proof}
		The first item follows directly from Theorem \ref{thm:collapse}. 
		It is sufficient to prove the second item assuming $G_{v_n}=K$ constant: 
		given a compact neighborhood of $Gv$, from \cite[Proposition IV.1.2]{bredon} there are only a finite number of orbit types touching it. That is, there is a finite number of subgroups $K_1,...,K_r$ whose conjugation classes are realized as isotropy subgroups of points in the neighborhood. In particular, for each $n$, $Gv_n\in \{G/K_1,...,G/K_r\}$.  It is then sufficient to prove Theorem \ref{prop:convergence} for each (infinite) subsequence of $v_n$ corresponding to  fixed orbit types. Taking these observations into account, we assume that $G_{v_n}=K$ is a fixed subgroup and  write $G_v=H$. A Slice Theorem argument guarantees that we can further choose $v_n$ such that $K<H$ (recall that $Gv$ has a neighborhood equivariantly diffeomorphic to  $G{\times_H} D$, where $D$ is a disc in $T_v({Gv})$; $G{\times_H} D$ is the quotient of $G\times D$ by the $H$-action given by left translation on $G$ and the isotropy representation on $D$; and $G$ acts on $G{\times_H} D$ by right translation on $G$.)
		
		
		The Proposition now follows from Theorem \ref{thm:collapse} and Corollary \ref{cor:Finsler} in \cite{ricci}: let $\g=\m\oplus\mathfrak k$ be a $(\, ,)$-orthogonal decomposition of $\g$, so that $\m$ is $K$-invariant. Consider the sequence of $K$-invariant metrics $g_n$ on $\m$ induced by the natural embeddings $i_n: G/K\to Gv_n$, $gK \mapsto g v_n$.
		Namely, let $X \in \g$ and denote by $X \cdot K$ the tangent vector of $G/K$ induced by $X$ at the basepoint $K$. Then $d i_n( X \cdot K ) = X v_n$ and, for $X, Y \in \m$, 
		$$
		g_n(X,Y) = \langle  d i_n( X \cdot K ), d i_n( Y \cdot K )  \rangle
		= \langle  X v_n, Y v_n \rangle.
		$$
		Since $v_n$ converges to $v$, $g_n$ converges to the (possibly degenerate) bilinear form $g$ defined  by the embedding $i:G/H\to Gv$. 
		By continuity, the metric induced by the form $g$, through the natural $K$-equivariant identification $\m/\m\cap \h \simeq \g/\h$, must coincide with the metric induced by the embedding  $i:G/H\to Gv$. 
		Since the action of $G$ is by isometries, $g$ must be $H$-invariant and, furthermore, the kernel of $g$ must be $\m\cap \h$.  Since $\g=\m\oplus\mathfrak k$ and $\mathfrak k \subseteq \h$, we get that $\h = \mathfrak{k} \oplus (\ker g)$ is a Lie algebra. Applying Theorem \ref{thm:collapse} and Corollary \ref{cor:Finsler}, we conclude that the induced metric spaces $(G/K,d_{g_n})$ converges to $(G/H,d_g)$.
		
		The third item follows in a similar way. Assume that there is a sequence $i_n$ and consider a subsequence with fixed isotropy group. Now the item follow since $\ker g=\ker i^*\tilde g=\g_v$ whenever $g$ is the metric induced by the embedding $i:G/K_v\to Gv$; and $g_n=i_n^*\tilde g$ converges uniformly to $g$.
	\end{proof}
	
The collapsing statements of Theorem \ref{ithm:figura} now follow:

\begin{corolario}\label{icorol:figura}
For the projected Ricci flow of Theorem \ref{ithm:figura} we have that
\begin{enumerate}
    \item The collapses $K, L, M$ are realized by isometric embeddings of the Ricci flow into the euclidean space $\g^2$.
    \item The collapses at $O,P,Q$ cannot be realized as equivariant isometric embeddings into a fixed Euclidean space.
\end{enumerate}
\end{corolario}
\begin{proof}
Item 1 is immediate from item 2 of Theorem \ref{prop:convergence} and the first statement in item 2 of Theorem \ref{ithm:figura}. For item 2, the collapses $O,P,Q$ satisfy the hypothesis of item 3 of Theorem \ref{ithm:convergence}, by Tables (33), (38) and (39) of \cite{ricci}.
\end{proof}

	In general, we do not know how vast is the class of flow lines that can be realized in an orbit space. But the result above guarantee that there is no smooth action $G\times V\to V$ where an entire line flow ending in $O,P$ or $Q$ is realized.
	
	On the other hand, \cite{moore1980equivariant} shows that any compact subset of a line flow can equivariantly embedded.
	
	\begin{teorema}[Moore--Schlafly \cite{moore1980equivariant}]\label{thm:schlafly}
		Let $M$ be a compact Riemannian manifold, possibly with boundary, and $G$ a compact Lie group acting on $M$ by isometries. Then, there is an orthogonal representation $\rho:G\to O(V)$ and an isometric equivariant embedding of $M$ into $(V,\rho)$.
	\end{teorema}
	
	Indeed, consider $\mathcal M$, the set of $G$-invariant metrics in $G/K$, as a subset of an Euclidean space with metric $g_0$. Given any compact subset $S\subset\mathcal M$, we can define the following Riemannian metric $\tilde g$ in $M_S=S\times G/K$:
	\[  \tilde g_{g,x}((V,X),(W,Y))=g_0(V,W)+g(X,Y),\]
	where $V,W\in T_gS$ and $X,Y\in T_xG/K$.
	Moreover, the action of $G$ on the second factor is  by $\tilde g$-isometries.
	In particular, Theorem \ref{thm:schlafly} applies and we conclude that  there is an Euclidean space $V$ and a representation $\rho:G\to O(V)$ where $S\times G/K$ is equivariantly and isometrically embedded. Since $\rho$ is orthogonal and $G$ is transitive in each slice $\{g\}\times G/K\subseteq M_S$, the embedding consists in a collection of $\rho$-orbits whose induced metrics are the corresponding metrics $g\in S$.  Although Theorem \ref{thm:embedding} presents one single embedding, not covering several regions that one could cover with Theorem \ref{thm:schlafly}, it explicitly realizes the singular points $K,L$ and $M$. Since boundary points are excluded from any compact set $S$, the existence of such a realization is not guaranteed by Theorem \ref{thm:schlafly}.
	
	As a final comment, we note that the classical Nash Embedding Theorem guarantees that there is an isometric embedding   $\varphi$ from the whole manifold $(M=\mathcal M\times G/K,\tilde g)$ to an Euclidean space $\R^N$. Since the restriction of $\varphi$ to each slice $\{g\}\times G/K$ is isometric, one might wish to explore Gromov--Hausdorff convergence through $\varphi$. Indeed, we argue that Hausdorff limits of embedded orbits coincides with their Gromov--Hausdorff limits. Denote $N_g=\{g\}\times G/K$.
	\begin{lema}
	    Let $g_n$ be as in Theorem \ref{thm:collapse}. Then,  the Hausdorff limit of $\varphi(N_{g_n})$ coincides with the Gromov--Hausdorff limit of $(G/K,d_g)$, up to isometry.
	\end{lema}
	\begin{proof} Consider a convergent sequence  $g_n$  of invariant metrics in $G/K$, $\lim g_n=g_\infty$, and the corresponding sequence of submanifolds $N_{g_n}=\{g_n\}\times G/K$. 
	
	Let $\overline M$ be a metric completion of $M$ and $\overline{\mathcal M}$ the closure of $\mathcal M$ in the space of bilinear forms in $\g$. Denote by $\pi:M\to \mathcal M$ the first coordinate projection and observe that $\pi$ is a \textit{submetry}, in the sense that $\pi$ is 1-Lipschitz and 
	\[d_{\tilde g}(\pi^{-1}(g),\pi^{-1}(g'))= d_{g_0}(g,g').\] 
	To conclude the equality, note that, fixed $x\in G/K$, the $d_{\tilde g}$-length of the line $\gamma(t)= ((1-t)g+tg',x)$ is
	\[|\dot\gamma(t)|_{(\gamma(t),x)}=g_0(g'-g,g'-g)^{1/2}=d_{g_0}(g,g'). \]

	In particular, we conclude  that $\pi$  extends to a surjective map $\bar\pi:\overline{M}\to \overline{\mathcal M}$. Moreover, since
	\[d_{GM}(N_{g_n},N_{g_\infty})\leq d_{\tilde g}(N_{g_n},N_{g_\infty}), \]
	the Gromov--Hausdorff of $N_{g_b}$ is $\bar\pi^{-1}(g_{\infty})$.
	
	Since  $\varphi$ is 1-Lipschitz, it can be extended as a 1-Lipschitz map  $\bar\varphi:\overline{ M}\to \R^N$. 
	We conclude that every $\epsilon$-neighborhood of $\bar\varphi(N_g)$ contains an infinite number of elements of $\{\varphi(N_{g_n})\}$, therefore, by the uniqueness of Gromov--Hausdorff limits, $\bar\varphi(N_g)$ must be isometric to  $N_g$.
	\end{proof}
	
	Although the last Lemma seems  expected, it is placed in an awkward position: $N_g$ is a length-space whose  distance  is possibly  not  Riemannian. On the other  hand, if a manifold $N$ is embedded into $\R^N$, (by approximating curves by smooth ones, one concludes that) its induced length-space structure is inherited from the induced Riemannian metric. It seems that, whenever the distance in $N_g$ is genuinely Finsler, the map $\bar\varphi|_{N_g}$ cannot be a smooth embedding.  It would be interesting to understand more deeply this phenomenon: either to prove that the Finsler norm  in Theorem \ref{thm:collapse} is always Riemannian or to have an example of an explicit realization of  $\bar\varphi(N_g)$  for a non-Riemnnian   $F_g$.

	%



\end{document}

%% file: retrato-fase.pdf_tex
\begingroup%
  \makeatletter%
  \providecommand\color[2][]{%
    \errmessage{(Inkscape) Color is used for the text in Inkscape, but the package 'color.sty' is not loaded}%
    \renewcommand\color[2][]{}%
  }%
  \providecommand\transparent[1]{%
    \errmessage{(Inkscape) Transparency is used (non-zero) for the text in Inkscape, but the package 'transparent.sty' is not loaded}%
    \renewcommand\transparent[1]{}%
  }%
  \providecommand\rotatebox[2]{#2}%
  \newcommand*\fsize{\dimexpr\f@size pt\relax}%
  \newcommand*\lineheight[1]{\fontsize{\fsize}{#1\fsize}\selectfont}%
  \ifx\svgwidth\undefined%
    \setlength{\unitlength}{693.79093326bp}%
    \ifx\svgscale\undefined%
      \relax%
    \else%
      \setlength{\unitlength}{\unitlength * \real{\svgscale}}%
    \fi%
  \else%
    \setlength{\unitlength}{\svgwidth}%
  \fi%
  \global\let\svgwidth\undefined%
  \global\let\svgscale\undefined%
  \makeatother%
  \begin{picture}(1,0.76077782)%
    \lineheight{1}%
    \setlength\tabcolsep{0pt}%
    \put(0,0){\includegraphics[width=\unitlength,page=1]{retrato-fase.pdf}}%
    \put(0.47975829,0.0051551){\color[rgb]{0,0,0}\makebox(0,0)[lt]{\lineheight{1.25}\smash{\begin{tabular}[t]{l}$O$\end{tabular}}}}%
    \put(0.48084932,0.73827247){\color[rgb]{0,0,0}\makebox(0,0)[lt]{\lineheight{1.25}\smash{\begin{tabular}[t]{l}$L$\end{tabular}}}}%
    \put(0.21827139,0.35364587){\color[rgb]{0,0,0}\makebox(0,0)[lt]{\lineheight{1.25}\smash{\begin{tabular}[t]{l}$K$\end{tabular}}}}%
    \put(0.71553935,0.35364587){\color[rgb]{0,0,0}\makebox(0,0)[lt]{\lineheight{1.25}\smash{\begin{tabular}[t]{l}$M$\end{tabular}}}}%
    \put(-0.00218174,0.70265889){\color[rgb]{0,0,0}\makebox(0,0)[lt]{\lineheight{1.25}\smash{\begin{tabular}[t]{l}$P$\end{tabular}}}}%
    \put(0.94262737,0.70265889){\color[rgb]{0,0,0}\makebox(0,0)[lt]{\lineheight{1.25}\smash{\begin{tabular}[t]{l}$Q$\end{tabular}}}}%
    \put(0.50172112,0.39034434){\color[rgb]{0,0,0}\makebox(0,0)[lt]{\lineheight{1.25}\smash{\begin{tabular}[t]{l}$R$\end{tabular}}}}%
    \put(0.50172112,0.50452768){\color[rgb]{0,0,0}\makebox(0,0)[lt]{\lineheight{1.25}\smash{\begin{tabular}[t]{l}$N$\end{tabular}}}}%
    \put(0.35902682,0.54839605){\color[rgb]{0,0,0}\makebox(0,0)[lt]{\lineheight{1.25}\smash{\begin{tabular}[t]{l}$T$\end{tabular}}}}%
    \put(0.59036454,0.54839605){\color[rgb]{0,0,0}\makebox(0,0)[lt]{\lineheight{1.25}\smash{\begin{tabular}[t]{l}$S$\end{tabular}}}}%
  \end{picture}%
\endgroup%